\documentclass[twoside,reqno,11pt]{fcaa-var}

\usepackage{graphicx}
\usepackage{epsfig}

\usepackage{amsthm}
\usepackage{amsmath}
\usepackage{latexsym}
\usepackage{amsfonts}
\usepackage{amssymb}

\newtheoremstyle{theorem}
  {15pt}          
  {15pt}  
  {\sl}  
  {\parindent}
  {\sc}  
  {. }   
  { }    
  {}     
\theoremstyle{theorem}
\newtheorem{lemma}{Lemma}[section]
\newtheorem{theorem}{Theorem}[section]

\newtheorem{proposition}{Proposition}[section]

\newtheoremstyle{defi}
  {15pt}          
  {15pt}  
  {\rm}  
  {\parindent}     
  {\sc}  
  {. }    
  { }    
  {}     
\theoremstyle{defi}
\newtheorem{definition}{Definition}[section]
\newtheorem{remark}{Remark}[section]

 \def\proofname{\indent\rm P\ r\ o\ o\ f.}
 
 \def\qed{\hfill$\Box$} 

\usepackage{hyperref}
\def\theequation{\arabic{section}.\arabic{equation}}
 
\usepackage{graphicx}
\usepackage{float}
\usepackage{braket}
\usepackage{bm}
\usepackage{mathrsfs}
\usepackage{mathtools,slashed}

\usepackage{changes}
\definechangesauthor[name=Eric, color=blue]{ed}
\definechangesauthor[name=Natalia, color=orange]{NR}

\usepackage[normalem]{ulem}
\usepackage{color}

\newcommand{\be}{\begin{eqnarray}}
\newcommand{\ee}{\end{eqnarray}}

\newcommand{\R}{\mathbb{R}}  
\newcommand{\N}{\mathbb{N}} 

\def\eu{\ensuremath{\mathrm{e}}}

\def\eg{{\it e.g.}}
\def\ie{{\it i.e.}}

\newcommand{\wh}[1]{\widehat{#1}}

\def\d{{\rm d}}

\def\themycopyrightfootnote{
   \hspace*{-0.7cm}
{\normalsize  \copyright \, 2020\,  Diogenes  Co., Sofia  \\
    \noindent  pp. ...--...\,, DOI: 10.1515/fca-2020-0002
    \hfill  \vspace*{-0.45cm}
\hspace*{10.4cm}
}}

 \title[ON THE FRACTIONAL LAPLACIAN \ldots]
 {ON THE FRACTIONAL LAPLACIAN OF\\ VARIABLE ORDER}
 \author[\normalsize E. Darve, M. D'Elia, R. Garrappa, A. Giusti, N. L. Rubio]
 {\normalsize Eric Darve$^1$, Marta D'Elia$^2$, Roberto Garrappa$^3$,\vspace*{2pt}\\ 
Andrea Giusti$^4$, Natalia L. Rubio$^5$ \vspace*{-4pt}}

\begin{document}

\vspace*{-1cm} 
\noindent
\par
\vspace*{-0.54cm} \vspace*{-28pt}
\hfill \hspace*{4cm}  

\medskip

\noindent \textbf{Short Note} \bigskip  

\begin{abstract}

We present a novel definition of variable-order fractional Laplacian on $\R^n$ based on a natural generalization {of the standard Riesz potential}. Our definition holds for values of the fractional parameter spanning the entire open set $(0, n/2)$. We then discuss some properties of the fractional Poisson's equation involving this operator and we compute the corresponding Green's function, for which we provide some instructive examples for specific problems. 

\medskip


\smallskip


\end{abstract}

\maketitle
\vspace*{-25pt}



\section{Introduction} \label{sec:intro}
\setcounter{section}{1} \setcounter{equation}{0}

The Laplace operator, or {\em Laplacian}, is a linear elliptic second-order ordinary differential operator defined on $\R^n$, with $n \in \N$, as the divergence of the gradient of a sufficiently regular function, {\ie}, $\triangle f(\bm{x}) := {\rm div} \big[ \bm{\nabla} f(\bm{x}) \big]$. This operator is ubiquitous in mathematical physics and finds applications ranging from wave dynamics, electrodynamics, gravity to biophysics and probability theory.

The increasing number of experimental observations of anomalous physical phenomena, {such as super- or sub-diffusion processes, recently motivated} the development of novel nonlocal models for these scenarios. These anomalous phenomena provide the perfect playground for fractional calculus, which is inherently a nonlocal theory, and, incidentally, provides us with several generalizations of the classical Laplacian. More precisely, the {\em fractional Laplacian} of order $s \in (0,1)$ can be defined in many equivalent ways on the whole $\R^n$ (see, {\em e.g.}, \cite{Lap-3}); however, once these definitions are restricted to bounded subsets of $\R^n$, they generally lead to different operators, thus spoiling the uniqueness of the definition of the fractional Laplacian. Yet, this is a very common trait of fractional operators \cite{FC1,FC2} and it does not represent a reason of concern. For a more detailed overview on the topic we refer the interested reader to \cite{Lap-1,Lap-2,Lap-3,Unified}.

In recent years, some attention has been devoted to the study of fractional models displaying a continuous transition between different fractional orders {to describe highly heterogeneous systems}. Such scenarios have motivated the development of {\em variable-order fractional operators}, for which the order itself becomes a function of time and/or space. This effort has therefore lead to the formulation of several inequivalent definitions of variable-order fractional derivatives and integrals {(see, {\em e.g.}, \cite{VOFD-1,VOFD-2,Felsinger2015,Schilling2015,DElia2021,AntilRautenberg2019}).}

Beside the non-uniqueness issue coming with variable-order fractional operators, there is also the problem of  providing a proper physical motivation for each of these definitions. In the largest part of the literature, the choice of the specific representation of these operators is often based on their aptness to reconcile the proposed mathematical model with experimental data. However, {this approach not only leads to representations that might be restricted to a limited range of fractional orders, disregarding physically-relevant cases}, but also lacks a proper justification that would follow from a bottom-up derivation of the mathematical model relying solely on basic principles of physics.

While there is no simple way of formulating a definition  of the fractional Laplacian as an operator in the real space for a general $s >0$, this complication disappears if such operator is 
defined starting from the Fourier space. Indeed, by means of the spectral representation one can define the fractional Laplacian 
for $0 < s < n/2$. This is done at the cost of trading an explicit real-space representation in favour of a larger range for $s$ taking advantage of a somewhat weaker definition for the operator. In contrast, the inverse fractional Laplacian is a well-defined operator in both real and Fourier spaces for all $s>0$. Clearly, a real space definition is essential if one wishes to extend these notions to 
the case of space-varying order. To this end we shall combine the definitions of inverse fractional Laplacian and the spectral representation to define the fractional Laplacian with $s=s(\bm{x})$. As a result we do not propose a direct definition of the variable-order fractional Laplacian in the real space since this is problematic when $1 < s < n/2$. Furthermore, since most of modern physics relies on the spectral decomposition of physical operators, this formulation of 
variable-order fractional Laplacian comes in handy for several potential applications. Indeed, it is worth mentioning that this approach finds part of its motivation in some recent results emerging from the theoretical study of galaxy rotation curves. In a nutshell, in \cite{MOND1, MOND2} it was argued that the phenomenology typically ascribed to dark matter can be understood as a modification of Newtonian gravity at galactic scales, thus without assuming the existence of an exotic and mysterious form of matter. This scenario then requires the field equation for the gravitational potential to become a variable-order fractional Poisson's equation with fractional order $s(\bm{x})$ ranging from $(1,3/2)$ in three spatial dimensions.

It is particularly relevant to note that the range $1 < s < n/2$ seems to be largely neglected in the literature on the variable-order fractional Laplacian and, to some extent, also for the case of the constant order one. This is mostly due to the fact that the vast majority of variable-order generalizations of this operator rely on real-space representations of the fractional Laplacian that typically force the restriction $0< s < 1$. For instance, this is the case of definitions that take advantage of the singular integral representation of the fractional Laplacian. The latter in fact is known to fail for $s \notin (0,1)$. We will come back to this point in Section \ref{sec:VO}.

Note that this study aims at providing a novel perspective on variable-order fractional models providing a new tool for potential physical applications. For this reason a rigorous mathematical analysis of the presented operators will be discussed elsewhere.  

\subsubsection*{Paper outline} This work is organized as follows. In Section \ref{sec:pre} we review some basic results on the standard fractional Laplacian, Riesz potential, and the fractional Poisson's equation. In Section \ref{sec:VO} we present a new definition of the variable-order fractional Laplacian and we derive the Green's function for the corresponding variable-order fractional Poisson's equation. In Section \ref{sec:example} we provide two instructive examples where we further discuss the properties of the Green's function of the variable-order fractional Poisson's equation on $\R^3$ for $0<s(\bm{x}) <1$ and $1<s(\bm{x}) <3/2$.

\vspace*{-3pt} 

\section{Preliminaries} \label{sec:pre}
\setcounter{section}{2} \setcounter{equation}{0}

\subsection{The classical Laplacian}

One of the key properties of the classical Laplacian is given by its {\em spectral representation}. Let $\mathcal{S} (\R^n)$ be the Schwartz space on $\R^n$, $f \in \mathcal{S} (\R^n)$, and $|\bm{k}|^2 := \bm{k}\cdot \bm{k}$. The Fourier transform of $f (\bm{x})$ is denoted by
$$
\wh{f} (\bm{k}) \equiv \mathcal{F} \left[ f (\bm{x}) \, ; \, \bm{k} \right] := 
\int _{\R^n} \eu ^{- i \, \bm{k}\cdot\bm{x}} \, f (\bm{x}) \, \d ^n x,
$$
so that
$$
\mathcal{F} \left[(-\triangle) f (\bm{x}) \, ; \, \bm{k} \right] = |\bm{k}|^2 \, \wh{f} (\bm{k}).
$$
Thus, the spectrum of the classical Laplacian is given by $\sigma (-\triangle) = [0, \infty)$.

Setting aside the Laplace equation, that constitutes the fundamental features of harmonic functions and plays an important role when seeking vacuum solutions for various physical problems, we focus on the {\em Poisson's equation}, \ie, 
\be
\label{eq:Poisson}
\triangle f(\bm{x}) = g(\bm{x}).
\ee
In this context, the spectral decomposition of the Laplacian provides an easy way to compute the fundamental solution of Eq. \eqref{eq:Poisson} through the method of Green's functions. Specifically, consider 
\be
\label{eq:PoissonGreen}
\triangle G (\bm{x}) = \delta (\bm{x}) \, ,
\ee
on the space of tempered distribution $\mathcal{S}' (\R^n)$, with $\delta (\bm{x})$ the Dirac delta in $n$-dimensions. For $n\geq 3$, by taking the Fourier transform of \eqref{eq:PoissonGreen}, solving for $\wh{G} (\bm{k})$, and inverting back to the coordinate space one finds
\be
G_n (\bm{x}) = - \frac{\Gamma \left(\frac{n}{2} - 1 \right)}{4 \pi ^{n/2}} \frac{1}{|\bm{x}|^{n-2}} \, ,
\ee
with $\Gamma (z)$ denoting Euler's gamma function. Note that the case $n=2$ displays a logarithmic behavior that can be rigorously derived exploiting the divergence theorem and Green's identities. Then, the fundamental solution of Eq. \eqref{eq:Poisson} reads
\be
f(\bm{x}) = (G_n \star g) (\bm{x}) \equiv \int _{\R ^n} G_n (\bm{x} - \bm{y}) \, g(\bm{y}) \, \d ^n y \, . 
\ee

\subsection{Fractional calculus}

Fractional calculus \cite{FC1,FC2,FC3} is a theory that provides a generalization of ordinary calculus based on {\em weakly singular} Volterra-like linear integro-differential operators. Within this theory one can provide a generalization of the notion of Laplacian following various routes (see, \eg, \cite{Lap-2}). The simplest definition of {\em fractional Laplacian}, as well as the least restrictive for the order of the resulting operator, relies on the spectral representation of the standard Laplacian, as we describe below.
\begin{definition}[Spectral representation] \label{def:spectralLap}
Let $f\in \mathcal{S} (\R^n)$ and $s \in (0, n/2)$. We define the fractional Laplacian $(-\triangle)^s$ as the linear operator such that
\begin{equation} \label{eq:spectralLap}
\mathcal{F} \left[(-\triangle)^s f (\bm{x}) \, ; \, \bm{k} \right] = |\bm{k}|^{2s} \, \wh{f} (\bm{k}) \, ,
\end{equation}
with $s$ denoting the (fractional) order of the operator.
\end{definition}

\begin{remark}
The results discussed in this work can be extended to a space of functions larger than $\mathcal{S} (\R^n)$. 
However, {the precise characterization of these spaces is not germane to the main message and motivation of this study. For details on this matter we refer the reader to, \eg, \cite{Lap-3}.}
\end{remark}

In the literature it is possible to find plenty of different representations of $(-\triangle)^s$ holding for $s \in U \subset (0, n/2)$, in most cases $U = (0,1)$. However, these alternative representations turn out to be equivalent to the spectral representation in Definition \ref{def:spectralLap} at least for some subsets of $(0, 1)$ and given some restrictions on the domain of applicability of the specific representation, see \cite{Lap-3}. 

Since we are interested in the fundamental solution of the {\em fractional Poisson's equation}, \ie,
\begin{equation}\label{eq:FractionalPoisson}
(-\triangle)^s f(\bm{x}) = - g(\bm{x}) \, ,
\end{equation}
it is natural to wonder about the existence of an ``inverse fractional Laplacian'' $(-\triangle)^{-s}$ that would allow to read-off the solution for any given source term $g(\bm{x})$. Indeed, such an operator exists and it is related to the well-known {\em Riesz potential} \cite{Riesz}, defined as follows.
\begin{definition}[Riesz potential] \label{def:RieszPotential}
Let $f\in L^1 _{\rm loc} (\R^n)$ and $\alpha \in (0, n)$, we define the Riesz potential $I_\alpha f$ as
\begin{equation}\label{eq:NotUsed1}
I_\alpha  f(\bm{x}) := \frac{\Gamma \left(\frac{n - \alpha}{2} \right)}{2^\alpha \pi^{\frac{n}{2}} \Gamma \left(\frac{\alpha}{2} \right)}
\int _{\R^n} \frac{f(\bm{y})}{|\bm{x} - \bm{y}|^{n-\alpha}} \, \d ^n y \, .
\end{equation} 
\end{definition}

By defining the kernel
\begin{equation}
\label{eq:KernelInverse}
K_s (\bm{x}) := \frac{\Gamma \left(\frac{n }{2} - s \right)}{4^s \pi^{\frac{n}{2}} \Gamma \left(s \right)}
\frac{1}{|\bm{x}|^{n-2 s}} \, ,
\end{equation}
we see that $I_{2s}  f(\bm{x}) = (K_s \star f) (\bm{x})$. 
Then, by noting that 
$$\widehat{K_s} (\bm{k}) = |\bm{k}|^{-2s}$$
one can prove the following theorem.
\begin{theorem}
Let $s \in (0, n/2)$, then
\be
(-\triangle)^s I_{2s} f(\bm{x}) =I_{2s}  (-\triangle)^s f(\bm{x}) = f(\bm{x})
\ee
for any $f \in \mathcal{S} (\R^n)$. In other words, $(-\triangle)^{-s} = I_{2s}$.
\end{theorem}

\vspace*{-3pt} 

\section{Variable-order fractional Laplacian} \label{sec:VO}
\setcounter{section}{3} \setcounter{equation}{0}

The standard approach to the variable-order fractional Laplacian relies on the {\em singular integral representation} of $(-\triangle)^s$, \ie,
\be
\label{eq:standardVO}
\nonumber
(-\triangle)^s f(\bm{x}) &=&
\frac{4^s \, \Gamma \left(\frac{n}{2} - s \right)}{\pi^{\frac{n}{2}} |\Gamma \left(-s\right)|} \, 
{\rm P.V.} \int _{\R^n} \frac{f(\bm{x}) - f(\bm{y})}{|\bm{x} - \bm{y}|^{n+2s}} \, \d ^n y\\
&=&
\lim _{\epsilon \to 0} \frac{4^s \, \Gamma \left(\frac{n}{2} - s \right)}{\pi^{\frac{n}{2}} |\Gamma \left(-s\right)|} \, 
\int _{\R^n \setminus B_{\epsilon} (\bm{x})} \frac{f(\bm{x}) - f(\bm{y})}{|\bm{x} - \bm{y}|^{n+2s}} \, \d ^n y
\ee
which is well defined for $s \in (0,1)$, with $B_{\epsilon} (\bm{x})$ denoting the ball radius $\epsilon$ centered at $\bm{x}$. The variable-order fractional Laplacian can then be obtained in several ways. The most straightforward extension consists in making $s$ a space-dependent quantity in Eq. \eqref{eq:standardVO}, see, e.g., \cite{Felsinger2015,Schilling2015}. Similarly to this approach, paper \cite{DElia2021} defines $s$ as a two-point space-dependent parameter, i.e., $s(\bm{x},\bm{y})$. Both these procedures have the caveat that the resulting variable-order operator is constrained by the condition $0<s(\bm{x})<1$, for a sufficiently regular $s(\bm{x})$, and do not span the whole set of values allowed by the spectral representation in Definition \ref{def:spectralLap}. Furthermore, this representation complicates both the analytical and numerical aspects of the study of fractional partial differential equations involving such an operator.

In this work we propose an alternative definition of variable-order fractional Laplacian with the purpose {of circumventing the issues} mentioned above. 
Before getting into the details of our proposal, it is convenient to introduce a few notions that will come in handy in what follows. Let us define the radial function 
\be
\label{eq:VOKernel}
{K}_{s(\cdot)} (\bm{x}) := 
\frac{\Gamma \left(\frac{n}{2} - s(|\bm{x}|) \right)}{4^{s (|\bm{x}|)} \pi^{\frac{n}{2}} \Gamma \left(s (|\bm{x}|) \right)}
\frac{1}{|\bm{x}|^{n-2 s (|\bm{x}|)}} \, ,
\ee
where $s(|\bm{x}|)$ is a sufficiently regular function with range $(0, n/2)$, then we have the following  definition.

\begin{definition}[Variable-order fractional Laplacian (VOFL)]
\label{def:VOFL} 
Let $s \in C^1[\R^n; (0, n/2)]$ such that
\begin{align*}
 \lim _{|\bm{x}| \to 0} s (\bm{x}) &= s_1 \in (0, n/2) \, , \\
 \lim _{|\bm{x}| \to \infty} s (\bm{x}) &= s_2 \in (0, n/2) \, ,
\end{align*}
and so that ${K}_{s(\cdot)}$ in Eq. \eqref{eq:VOKernel} has a non-vanishing, sufficiently regular Fourier transform $\widehat{K}_{s(\cdot)} (\bm{k})$ for which $\wh{f} (\bm{k}) / \widehat{K}_{s(\cdot)} (\bm{k})$ admits an inverse Fourier transform for any $f \in \mathcal{S} (\R^n)$.  Then, the variable-order fractional Laplacian $(- \triangle)^{s(\bm{\cdot})}$ is defined as the operator satisfying the condition
\be
\label{eq:definition-spectrum}
\mathcal{F} \left[(- \triangle)^{s(\bm{\cdot})} f (\bm{x}) \, ; \, \bm{k} \right] = \frac{\wh{f} (\bm{k})}{\widehat{K}_{s(\cdot)} (\bm{k})} \, .
\ee
\end{definition}

\begin{remark}
We stress the fact that, compared to the variable-order definitions mentioned above, here, $s$ is a radial function. {The reasons of this choice will be clear in Remark \ref{re:radial}.}
\end{remark}

\begin{remark}
The continuity condition and the requirements on the asymptotic behavior of $s(\bm{x})$ tell us that ${K}_{s(\cdot)} \in L ^1 _{\rm loc} (\R^n)$ and it decays at infinity as $1/|\bm{x}|^{n-2 s_2}$. This means that ${K}_{s(\cdot)} (\bm{x})$ can be Fourier-transformed and $\widehat{K}_{s(\cdot)} \in \mathcal{S}'(\R^n)$. 
\end{remark}

Furthermore, if we define a variable-order Riesz potential $I_{2 s(\cdot)}$ as
\be
\label{eq:voriesz}
I_{2 s(\cdot)} f(\bm{x}) := \left( {K}_{s(\cdot)} \star f \right) (\bm{x}) \, ,
\ee
for a sufficiently regular function $f(\bm{x})$ one can show that
\be
(- \triangle)^{s(\cdot)} I_{2 s(\cdot)} f(\bm{x}) = I_{2 s(\cdot)} (- \triangle)^{s(\cdot)} f(\bm{x}) = f(\bm{x}) \, , 
\ee
taking advantage of the spectral representation of $I_{2 s(\cdot)}$ and of Definition \ref{def:VOFL}. Note that in order to identify the variable-order Riesz potential as the {\em inverse} of the VOFL these two operators have to be bijections between function spaces. A precise characterization of the function spaces involved in Definition \ref{def:VOFL} is beyond the scope of this Letter and it is left to future studies.

This new definition of VOFL extensively simplifies the analytical treatment of problems involving this operator, though they still remain somewhat problematic at the numerical level when $s \ge 1$. For instance, with the operator in Definition \ref{def:VOFL} one can compute the Green's function for the variable-order Poisson's equation
\be
\label{eq:vopoisson}
(- \triangle)^{s(\cdot)} \Phi (\bm{x}) = - \delta (\bm{x}) \, .
\ee
Indeed, taking the Fourier transform of both sides of the latter equation and multiplying by 
$\widehat{K}_{s(\cdot)} (\bm{k})$ on gets
\be
\wh{\Phi} (\bm{k}) = - \widehat{K}_{s(\cdot)} (\bm{k}) \, .
\ee
Then, computing the inverse Fourier transform one finds
\be
\Phi (\bm{x}) = - {K}_{s(\cdot)} (\bm{x}) = - \frac{\Gamma \left(\frac{n}{2} - s(|\bm{x}|) \right)}{4^{s (|\bm{x}|)} \pi^{\frac{n}{2}} \Gamma \left(s (|\bm{x}|) \right)}
\frac{1}{|\bm{x}|^{n-2 s (|\bm{x}|)}} \, .
\ee
Hence, given a source function $g(\bm{x})$ the solution of
\be
(-\triangle)^{s(\cdot)} f(\bm{x}) = - g(\bm{x}) \, ,
\ee
simply reads
\be
f(\bm{x}) = \left( \Phi \star g \right) (\bm{x}) = -\left( K_{s(\cdot)} \star g \right) (\bm{x}) \, .
\ee

\begin{remark}\label{re:radial}
Note that the kernel in Eq. \eqref{eq:VOKernel} is a radial function, \ie, it depends solely on $|\bm{x}|$. This is choice aims at preserving the invariance under rotations of the Green's function for the proposed VOFL. In other words, we required the Green's function for the variable-order Poisson's equation to be scalar (i.e., invariant) under rotations.
\end{remark}

Finally, it is worth pointing out that the specific choice of \eqref{eq:VOKernel} and the assumptions in Definition \ref{def:VOFL} select ({\em a priori}) a specific asymptotic behaviour for the Green's function of \eqref{eq:vopoisson}, namely both $\Phi (\bm{x})$ and $|\bm{\nabla}\Phi (\bm{x})|$ vanish as $|\bm{x}|\to \infty$.

\vspace*{-3pt} 

\section{{Two instructive examples}} \label{sec:example}
\setcounter{section}{4} \setcounter{equation}{0}

\subsection{First example}\label{sec:firstexample}
We consider the fundamental solution, in $n=3$ dimensions, of Eq. \eqref{eq:vopoisson} with
\be
\label{eq:firstexample}
s(r) = \frac{6 + 9 \, r}{10(1 +  r)} \, ,
\ee
where $r=|\bm{x}|$. Then $s\in C^1[\R^3;(0.6,0.9)]$ and
\[
\begin{aligned}
& s (r) = 6/10 + \mathcal{O} (r) \, , & &\mbox{as} \,\, r \to 0 \, , \\
& s (r) = 9/10 + \mathcal{O} (r^{-1}) \, , & &\mbox{as} \,\, r \to \infty \, ,
\end{aligned}
\]
satisfying the conditions on $s(r)$ in Definition \ref{def:VOFL} and granting that ${K}_{s(\cdot)} \in L^1 _{\rm loc} (\R^3)$. 
Furthermore, {we can prove the following lemma.}
\begin{lemma}
\label{lemma-example}
The function $p (r) = r {K}_{s(\cdot)} (r)$ is positive decreasing function on $r>0$ such that
\[
\lim _{r \to 0} p (r) = + \infty \, , \qquad \lim _{r \to \infty} p (r) = 0 \, .
\]
\end{lemma}
\begin{proof}
First, $p (r)$ is positive since it is defined as the product of two positive functions. Second, the two limits are obtained by studying the asymptotic behavior of $p (r)$, {\em i.e.},
\begin{equation*}
\begin{split}
p (r) &=
\frac{\Gamma \left(\frac{9}{10}\right)}{2^{6/5} \pi ^{3/2} \, \Gamma \left(\frac{3}{5}\right) \, r^{4/5}}
- \frac{3 \Gamma \left(\frac{9}{10}\right) r^{1/5}}{10 \cdot {2}^{6/5} \, \pi ^{3/2} \, \Gamma \left(\frac{3}{5}\right)} \times \\
& \qquad \times  
\left[-2 \log r+\log 4+\psi \left(\frac{3}{5}\right)+\psi \left(\frac{9}{10}\right)\right]+\mathcal{O}(r^{6/5}) \\
& \qquad \mbox{as} \,\, r \to 0 \, , \\
p (r) &= \frac{\Gamma \left(\frac{3}{5}\right)}{2^{9/5} \, \pi ^{3/2} \, \Gamma \left(\frac{9}{10}\right)} \, 
\frac{1}{r^{1/5}} +
\mathcal{O}\left(r^{-6/5}\right) \\
& \qquad \mbox{as} \,\, r \to \infty \, .
\end{split}
\end{equation*}
with $\psi (z) := \frac{\d }{\d z } \log \Gamma (z)$ denoting the digamma function. Finally, in order to prove the monotonicity of $p (r)$, we compute its first derivative with respect to $r$, here denoted with the prime:
\begin{eqnarray*}
p'(r) &\!\!=\!\!&
-\frac{{K}_{s(\cdot)}(r)}{10 (1+r)^2} \Bigg\{
2 r (r+5+\log 8)-6 r \log r + 8  \\
& &+ 3 r \left[\psi \left(\frac{9 r+6}{10( r+1)}\right)+\psi \left(\frac{3}{10} \left(2+\frac{1}{r+1}\right)\right)\right]\Bigg\}  \\
\end{eqnarray*}
By taking advantage of the monotonicity of the digamma function (see the Bohr-Mollerup Theorem in \cite{BMTh}, and, also, \cite{WW}) we have proved that $p'(r) < 0$ for all $r>0$.
\end{proof}

One can then use this result to prove the following proposition. 
\begin{proposition}
The function ${K}_{s(\cdot)} (r)$ with $s(r)$ as in Eq. \eqref{eq:firstexample} has a positive Fourier transform. 
\end{proposition}

\begin{proof}
First, let us recall that the Fourier transform of a radial function, in three space dimensions, reads
\begin{eqnarray*}
\wh{K}_{s(\cdot)} (k) &=& \frac{4 \pi}{k} \int_0 ^\infty r {K}_{s(\cdot)} (r) \, \sin(k r) 
\, \d r \\
&=& \frac{4 \pi}{k} \int_0 ^\infty p(r) \, \sin(k r) \, \d r.
\end{eqnarray*}
In other words, computing $\wh{K}_{s(\cdot)} (k)$ corresponds to computing the Fourier-Sine transform of $p(r)$ defined as in Lemma \ref{lemma-example}. Since $p(r)$ is a positive, decreasing function of $r$ with a weak singularity at $r=0$ and vanishing as $r\to \infty$, its Fourier-Sine transform is positive \cite{positivity}, thus implying $\wh{K}_{s(\cdot)} (k) > 0$ for $k>0$. 
\end{proof}

This result shows that $\wh{K}_{s(\cdot)} (k)$ does not vanish on $k>0$, thus proving explicitly that $K_{s(\cdot)} (k)$ satisfies the requirements in Definition \ref{def:VOFL}. {In Figure \ref{fig:example_plot1} we report plots of $\widehat K_{s(\cdot)} (k)$} and of the fundamental solution of the variable-order Poisson's equation ({\em i.e.}, the solution of \eqref{eq:vopoisson}) with $s(r)$ as in \eqref{eq:firstexample}, {together with the two limiting cases $s=0.6$ and $s=0.9$.}

\begin{figure}[htbp]
    \centering
   \includegraphics[scale=0.7]{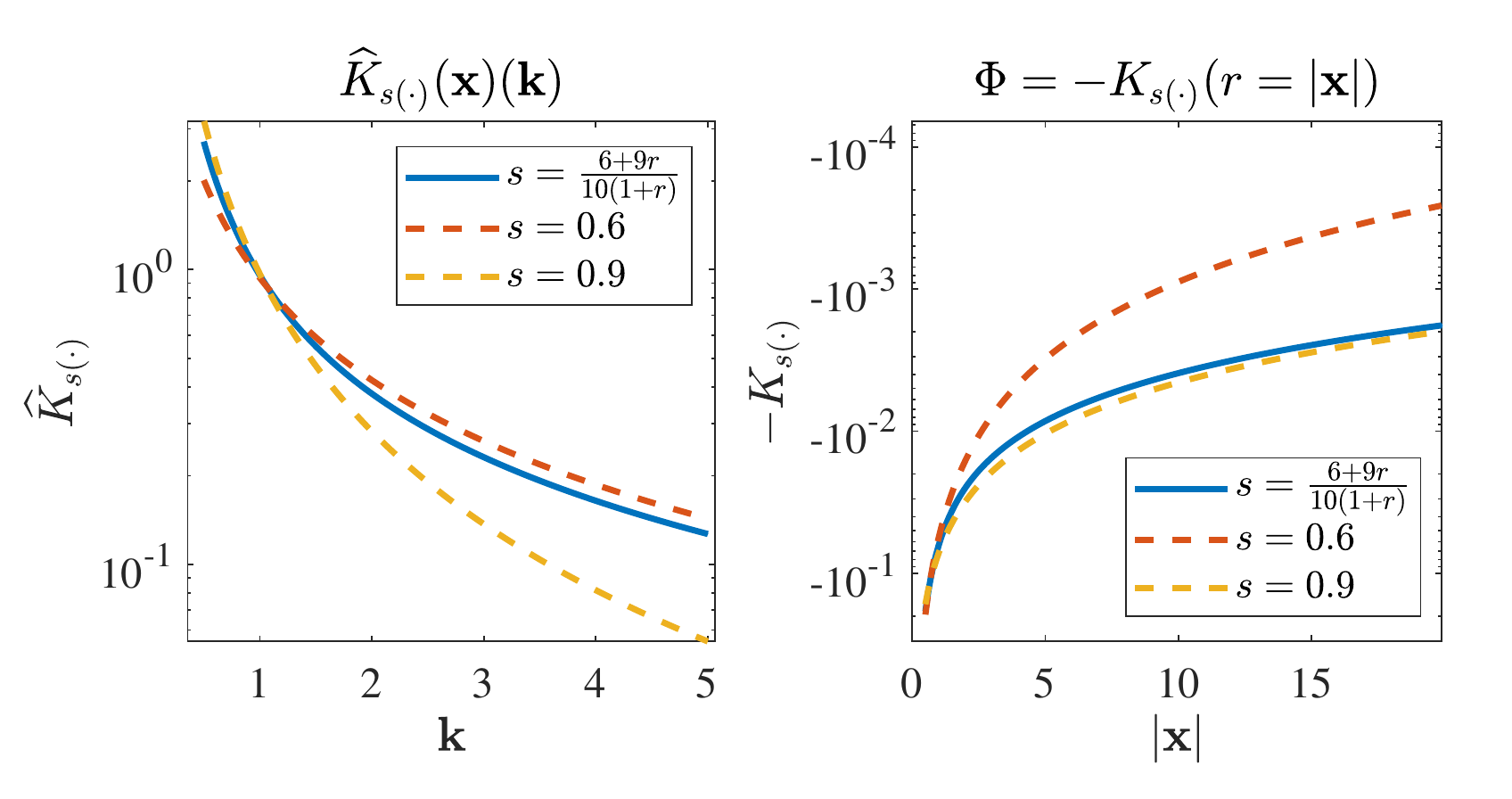}
   \caption{Plots of $\widehat{K}_{s(\cdot)} (\bm{x}) (\bm{k})$ and $\Phi = - {K}_{s(\cdot)}(r=|\bm{x}|)$}
   \label{fig:example_plot1}
\end{figure}

\subsection{Second example}

We consider the same setting as in Section \ref{sec:firstexample}, with
\begin{equation}\label{eq:secondexample}
s(r) = \frac{1}{10} \frac{11+13 r}{1 +  r} \, ,
\end{equation}
which, again, belong to $C^1[\R^3; (1.1, 1.3)]$ and such that
\[
\begin{aligned}
& s (r) = 11/10 + \mathcal{O} (r) \, ,  &  &\quad \mbox{as} \,\, r \to 0 \, , \\
& s (r) = 13/10 + \mathcal{O} (r^{-1}) \, ,  &  &\quad \mbox{as} \,\, r \to \infty \, ,
\end{aligned}
\]
{\em i.e.}, $1<s(r)<3/2$ for $r > 0$, thus satisfying the conditions on $s(r)$ in Definition \ref{def:VOFL} and yielding ${K}_{s(\cdot)} \in L^1 _{\rm loc} (\R^3)$. 

This case is more difficult to compute numerically because the Fourier transform is only defined as a generalized function or distribution. As indicated previously, using radial symmetry in 3D, we use:
\[
\wh{K}_{s(\cdot)} (k) = \frac{4 \pi}{k} \int_0^\infty r {K}_{s(\cdot)} (r) \, \sin(k r) 
\, \d r.
\]
This integral is defined in the classical sense only when $r {K}_{s(\cdot)} (r) \to 0$ for $r \to \infty$. This was the case in Example 1 but for Example 2, we need to consider the following definition:
\[
\wh{K}_{s(\cdot)} (k) = 
\lim_{\lambda \to 0+} \frac{4 \pi}{k} \int_0^\infty e^{-\lambda r} r {K}_{s(\cdot)} (r) \, \sin(k r) 
\, \d r.
\]
To estimate this integral numerically, we partition it using the period of $\sin(kr)$:
\[
\int_0^\infty e^{-\lambda r} r {K}_{s(\cdot)} (r) \, \sin(k r) 
\, \d r
= \sum_{i=0}^\infty
\int_{i2\pi/k}^{(i+1)2\pi/k} e^{-\lambda r} r {K}_{s(\cdot)} (r) \, \sin(k r) 
\, \d r. 
\]
For $r$ sufficiently large, we have:
\[
e^{-\lambda r} r {K}_{s(\cdot)} (r)
\propto
e^{-\lambda r} r^a, 
\]
with $a = 1-n+2s(\infty) = 0.6$ (in this example). We can then estimate the integral over one period of $\sin(kr)$ (for $i$ sufficiently large):
\begin{align}
& \Delta \hat{K}_{s(\cdot),\;i} = \int_{i2\pi/k}^{(i+1)2\pi/k} e^{-\lambda r} r {K}_{s(\cdot)} (r) \, \sin(k r) 
\, \d r = \notag \\
& \hspace*{40pt} \frac{8 \pi^2}{k^3}
\big(\lambda - \frac{a}{r_i}\big) e^{-\lambda r_i} r_i {K}_{s(\cdot)}(r_i) + O(e^{-\lambda r_i} k^{-4}), \label{eq:dki} \\
& r_i = (i+1/2) \frac{2\pi}{k}. \notag
\end{align}
We see that the sum over $i$ is convergent as expected. In Figure \ref{fig:khat_dkhat_aug.pdf}, we show $\widehat{K}_{s(\cdot)}$ and the value of $\Delta \widehat{K}_{s(\cdot),\;i}$ vs the period number $i$.

\begin{figure}[htbp]
  \centering
    \includegraphics[scale=0.7]{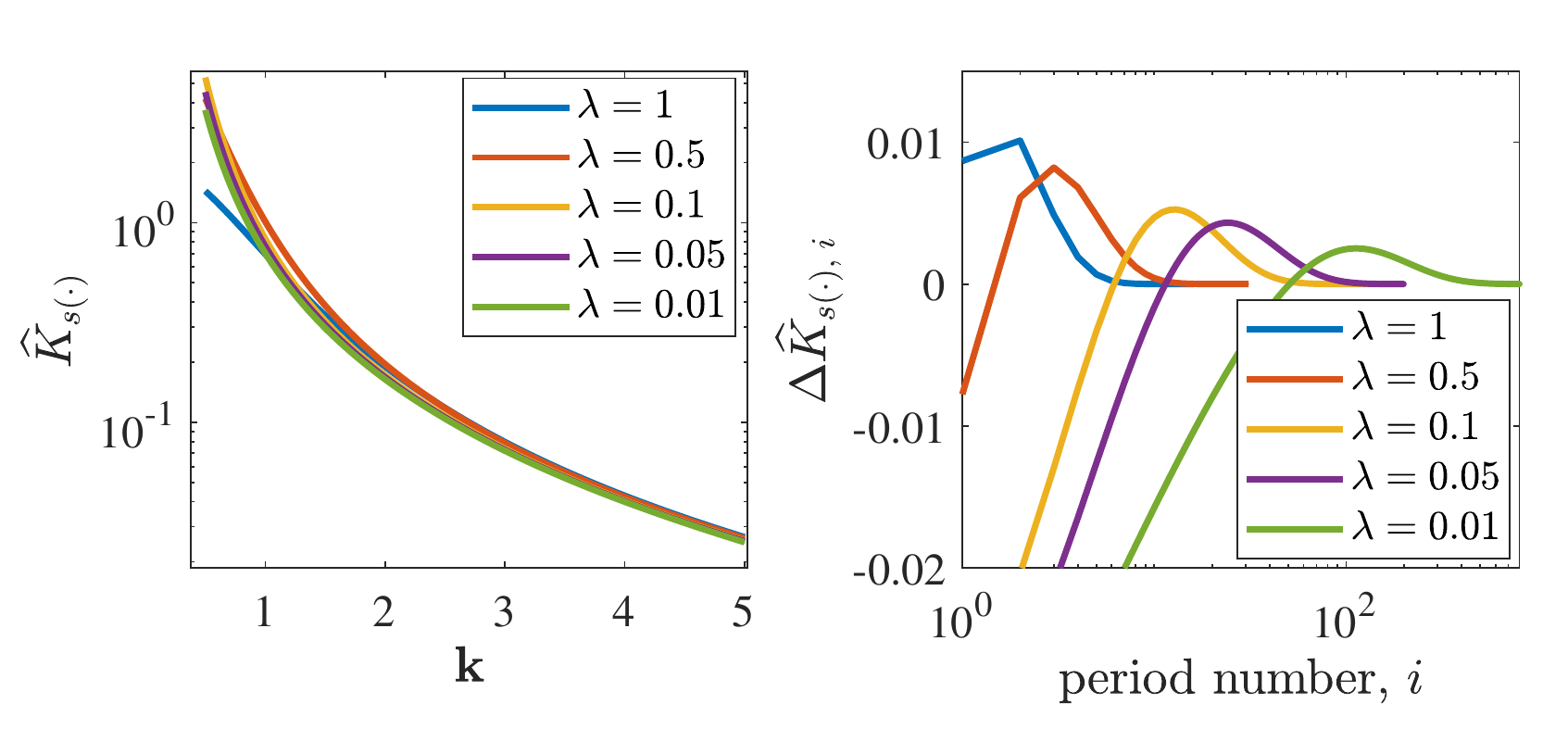}
    \caption{Plots of $\widehat{K}_{s(\cdot)} $ for different values of $\lambda$ vs the frequency $k$ (left) and $\Delta \widehat{K}_{s(\cdot),\;i}$ for $k=5$ for different values of $\lambda$ vs the period number $i$ (right). As expected the convergence is slower for small $\lambda$. The positive local maximum for $\Delta \widehat{K}_{s(\cdot),\;i}$ occurs around $i \approx a k\lambda^{-1}$.}
    \label{fig:khat_dkhat_aug.pdf}
\end{figure}

Using this analysis, we can estimate the error when numerically truncating the sum over $i$. We can use the following estimate for $I$ sufficiently large:
\begin{equation}\label{eq:err}
\Big | \int_{2\pi I/k}^{\infty} e^{-\lambda r} r {K}_{s(\cdot)}(r) \, \sin(k r) 
\, \d r
\Big| = \frac{e^{-\lambda r_I} r_I {K}_{s(\cdot)}(r_I)}{k} + O(e^{-\lambda r_I} k^{-2}),
\end{equation}
with $r_I = 2\pi I / k$.

In Figure \ref{fig:khat_error_aug} we show $\Delta \widehat{K}_{s(\cdot),\;i}$, computed numerically and using Eq.~\eqref{eq:dki}, and the truncation error for $\widehat{K}_{s(\cdot)}$ vs the period $i$ computed numerically and with Eq.~\eqref{eq:err}.

\begin{figure}[htbp]
  \centering
    \includegraphics[scale=0.7]{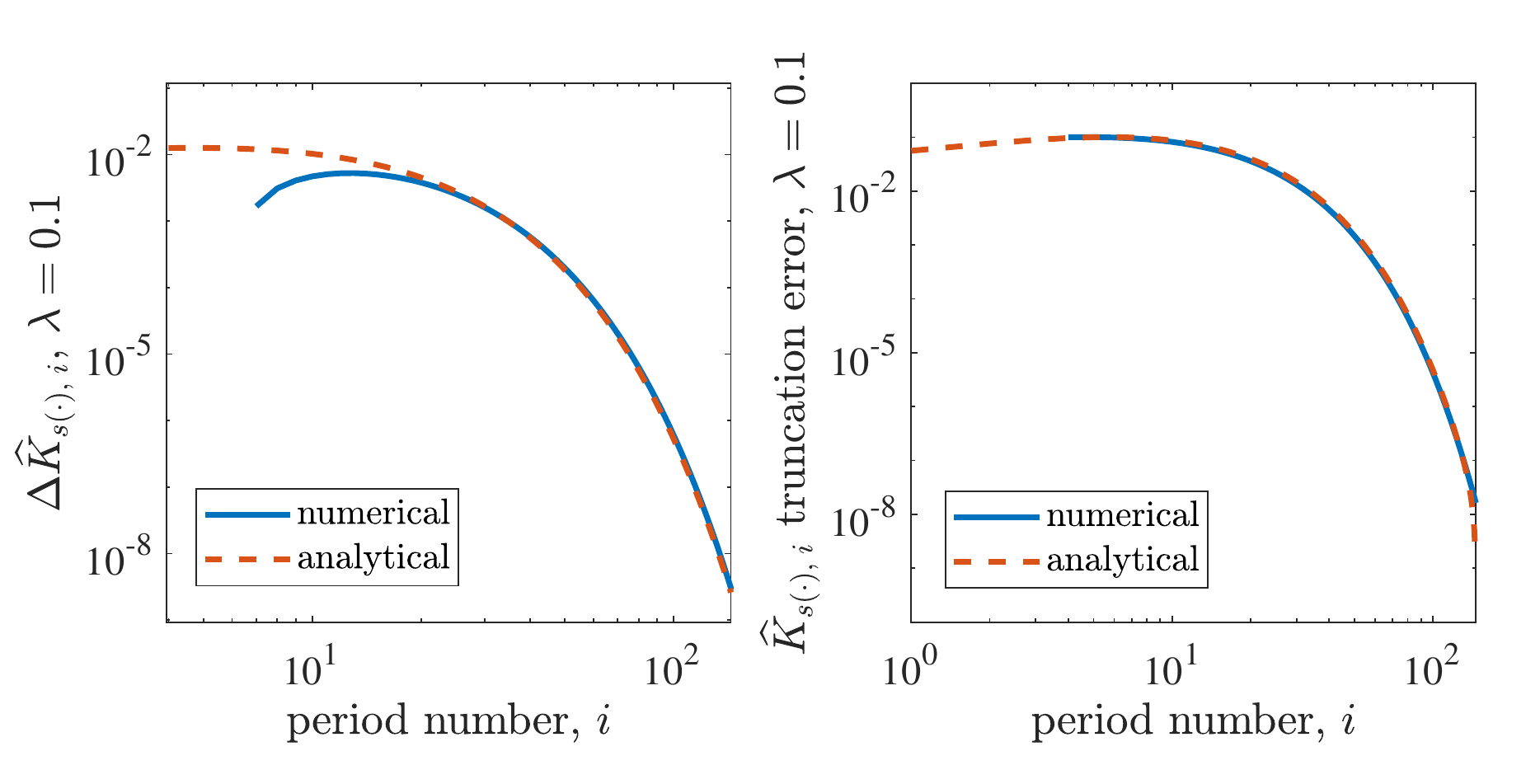}
    \caption{Plots of $\Delta \widehat{K}_{s(\cdot),\;i}$ computed numerically and with Eq.~\eqref{eq:dki} (left) and the truncation error estimated numerically and with Eq.~\eqref{eq:err} (right) for $k=5$.}
    \label{fig:khat_error_aug}
\end{figure}

When $s(\infty) < 1$, $a<0$ and the integral converges even when $\lambda = 0$. This was the case in Example 1.

For $s(\infty) > 1$ and $a > 0$, the convergence with respect to $\lambda \to 0+$ is more difficult to establish and we will rely simply on the numerical benchmarks for that.

In Figure \ref{fig:khat_bounds_aug}, we report plots of $\widehat K_{s(\cdot)} (k)$ and of the fundamental solution of the variable-order Poisson's equation with $s(r)$ as in Eq.~\eqref{eq:firstexample}, together with the two limiting cases $s=1.1$ and $s=1.3$.

\begin{figure}[htbp]
  \centering
    \includegraphics[scale=0.7]{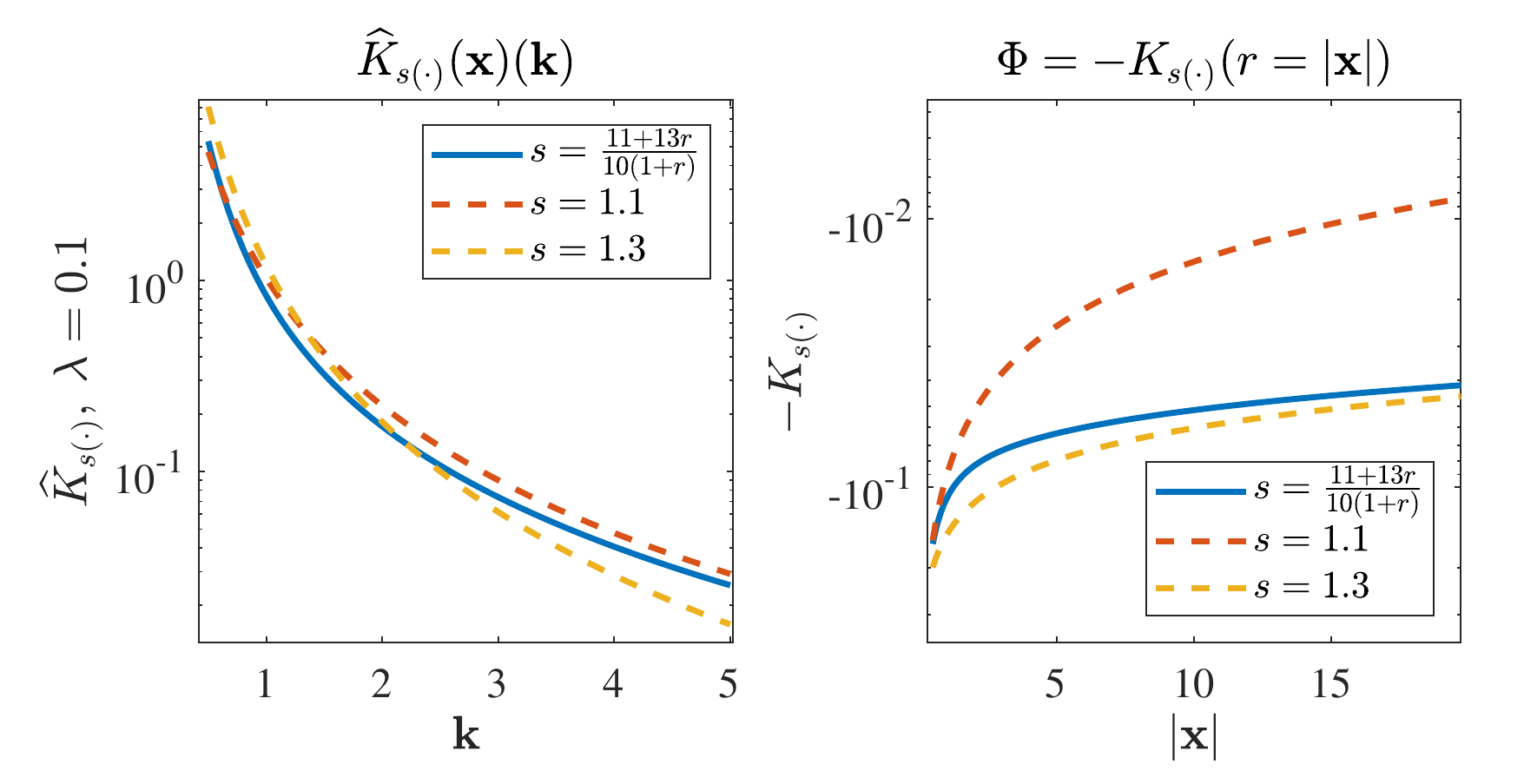}
    \caption{Plots of $\hat{K}_{s(\cdot)} (\bm{x}) (\bm{k})$ and $\Phi = - {K}_{s(\cdot)}(r=|\bm{x}|)$.}
    \label{fig:khat_bounds_aug}
\end{figure}


\vspace*{-3pt} 

\section{Discussion and outlook} \label{sec:conclusions}
\setcounter{section}{4} \setcounter{equation}{0}

We presented a novel definition of the variable-order fractional Laplacian based on a natural generalization {of the constant-order Riesz potential. A remarkable feature of this definition} is that it extends the range of the fractional parameter $s(\bm{x})$ from $(0,1)$ to the entire set $(0, n/2)$, with $n$ denoting the number of spatial dimensions. Also, we investigated the properties of the generalized Poisson's equation involving this operator and we computed its Green's function. Lastly, we provided two explicit realizations of this Green's function, specifically for $0.6<s(\bm{x})<0.9$ and $1.1<s(\bm{x})<1.3$. Although for the first example we can rigorously check that
${K}_{s(\cdot)} (\bm{x})$ satisfies the requirements of Definition \ref{def:VOFL}, this is not the case for the second example and one can only verify these conditions numerically.

This novel definition provides a valuable tool for studying various physical scenarios displaying nonlocal properties. As an example, the VOFL as in Definition \ref{def:VOFL} may provide insights for validating fractional Newtonian gravity \cite{MOND1,MOND2} against astrophysical observations. In other words, by using a regression procedure based on observations, one could identify the optimal $\widehat{K}$, from which $K$ can be then computed.

It is worth pointing out that this work is intended as a ``proof of concept'' aimed at paving a new way for the study of variable-order fractional problems on $\R^n$. This approach thus begs for further investigations in the realm of function spaces and, more in general, for a more rigorous underlying mathematical structure. This is, however, not the scope of this study and therefore this matter is left to future works. 

\vspace*{-2pt}

\section*{Acknowledgments}	

A. Giusti is supported by the European Union's Horizon 2020 research and innovation programme under the Marie Sk\l{}odowska-Curie Actions (grant agreement No. 895648 -- CosmoDEC). His work has been carried out in the framework of the activities of the Italian National Group for Mathematical Physics [Gruppo Nazionale per la Fisica Matematica (GNFM), Istituto Nazionale di Alta Matematica (INdAM)]. 

R. Garrappa is partially supported by the COST Action CA 15225 ``Fractional-order systems-analysis, synthesis and their importance for future design'' and by the Italian National Group for Scientific Computing (INdAM-GNCS) through an INdAM-GNCS 2020 project. 

E. Darve and M. D'Elia are partially supported by the U.S. Department of Energy, Office of Advanced Scientific Computing Research under the Collaboratory on Mathematics and Physics-Informed Learning Machines for Multiscale and Multiphysics Problems (PhILMs) project (DE-SC0019453). 

M. D'Elia is also supported by the Sandia National Laboratories (SNL) Laboratory-directed Research and Development program. SNL is a multimission laboratory managed and operated by National Technology and Engineering Solutions of Sandia, LLC., a wholly owned subsidiary of Honeywell International, Inc., for the U.S. Department of Energy's National Nuclear Security Administration under contract {DE-NA0003525}. This paper, SAND2021-10850 R, describes objective technical results and analysis. Any subjective views or opinions that might be expressed in this paper do not necessarily represent the views of the U.S. Department of Energy or the United States Government.


\renewcommand{\bibliofont}{\normalsize}
\bibliographystyle{fcaa-like-bib-style}
\bibliography{Biblio}


 \newpage 

 \it

 \noindent $^1$ Institute for Computational and Mathematical Engineering,\\
Stanford University,\\
Stanford, CA, 94305, USA\\
and\\
Department of Mechanical Engineering,\\
Stanford University,\\
Stanford, CA, 94305, USA\\[4pt]
$^2$ Computational Science and Analysis,\\
Sandia National Laboratories,\\
Livermore, CA, USA\\[4pt]
$^3$ Department of Mathematics\\
 University of Bari\\
 Via E. Orabona 4, 70126 Bari, ITALY\\
 and\\
 the INdAM Research group GNCS\\	
 e-mail: roberto.garrappa@uniba.it\\[4pt]
$^4$ 
 Institute for Theoretical Physics, \\
 ETH Zurich,\\
 Wolfgang-Pauli-Strasse 27,\\
 8093, Zurich, Switzerland\\
 e-mail: agiusti@ubishops.ca
 \hfill Received:
 \\[4pt]
 $^5$ Stanford University,\\
Stanford, CA, 94305, USA
\\[4pt] 

\rm

\medskip

\vskip 0.05cm  
\hrule width40mm height0.10mm 
\vskip 0.05cm

\end{document}